\documentclass[twoside,11pt,leqno]{article}

\usepackage{itjournal2}
\usepackage{amssymb}
\usepackage{amsmath}
\usepackage{amsfonts}
\usepackage{enumerate}
\usepackage{color}
\usepackage[numbers,sort&compress]{natbib}

\usepackage{amsthm}

\newtheorem{theorem}{Theorem}[section]

\newtheorem{lemma}{Lemma}[section]

\newtheorem{corollary}{Corollary}[section]

\let\oldproofname=\proofname
\renewcommand{\proofname}{\rm\bf{\oldproofname}}

\usepackage[all]{xy}

\itjournalheading{N.}{...--}{... (1--8)}

\ShortHeadings{Semigroup of transformations with restricted partial range: Regularity ...}{Jiulin Jin and Taijie You}
\firstpageno{1}

\begin{document}

\title{Semigroup of transformations with restricted partial range: Regularity, abundance and some combinatorial results}

\author{\name Jiulin Jin\\
       \addr College of Mathematics and Information Science\\
             Guiyang University\\
             Guiyang, Guizhou, 550005\\
             China
             \\\email{j.l.jin@hotmail.com}
             \AND \name Taijie You
       \\\addr College of Mathematical Sciences \\
       Guizhou Normal University\\
       Guiyang, Guizhou, 550025\\
             China
       \\\email{youtaijie1111@163.com}}


\maketitle

\begin{abstract}
Suppose that $X$ be a nonempty set. Denote by $\mathcal{T}(X)$ the full transformation semigroup on $X$. For $\varnothing \neq Z\subseteq Y\subseteq X$, let $\mathcal{T}(X,Y,Z)=\{\alpha \in \mathcal{T}(X): Y\alpha \subseteq Z \}$. Then $\mathcal{T}(X,Y,Z)$ is a subsemigroup of $\mathcal{T}(X)$. In this paper, we characterize the regular elements of the semigroup $\mathcal{T}(X,Y,Z)$, and present a necessary and sufficient condition under which $\mathcal{T}(X,Y,Z)$ is regular. Furthermore, we investigate the abundance of the semigroup $\mathcal{T}(X,Y,Z)$ for the case $Z\subsetneq Y\subsetneq X$. In addition, we compute the cardinalities of $\mathcal{T}(X,Y,Z)$, ${\rm Reg}(\mathcal{T}(X,Y,Z))$ and ${\rm E}(\mathcal{T}(X,Y,Z))$ when $X$ is finite, respectively.
\end{abstract}


\section{Introduction}\label{intro}

An element $x$ of semigroup $S$ is called a \emph {regular element} of $S$ if $x=xyx$ for some $y\in S$, and $S$ is said to be a \emph {regular semigroup} if every element of $S$ is regular. The set of all regular elements of a semigroup $S$ is denoted by ${\rm Reg}(S)$. An element $x$ of semigroup $S$ is called an \emph {idempotent} of $S$ if $x^2=x$. The set of all idempotents of a semigroup $S$ is denoted by ${\rm E}(S)$. Regular element (resp., idempotent) is one of the most studied topics in semigroup theory due to its nice algebraic properties and wide applications. There have been many research works studying regularity of semigroups (see \cite{Araujo,J,K,M1,N,P,S1,S2,Y}).

Let $S$ be a semigroup and $a, b\in S$. We say that $a$ and $b$ are \emph {$\mathcal{L}$-related} (\emph {$\mathcal{R}$-related}) in $S$ if $S^1a=S^1b$ ($aS^1=bS^1$) where $S^1$ denotes the monoid obtained from $S$ by adding an identity if $S$ has no identity, otherwise, $S^1 = S$. If $a$ and $b$ are $\mathcal{L}$-related ($\mathcal{R}$-related), we can write $(a,b)\in \mathcal{L}$ ($(a,b)\in \mathcal{R}$). Again, if $(a,b)\in \mathcal{L}$ in some oversemigroup of $S$, then $a$ and $b$ are called \emph {$\mathcal{L}^{\ast}$-related} and write $(a,b)\in \mathcal{L}^{\ast}$. The relation $\mathcal{R}^{\ast}$ can be defined dually. Clearly, $\mathcal{L}\subseteq \mathcal{L}^{\ast}$, $\mathcal{R}\subseteq \mathcal{R}^{\ast}$, and $\mathcal{L}^{\ast}$ and $\mathcal{R}^{\ast}$ are equivalence relations on $S$. Fountain \cite{Fountain} pointed out that a semigroup $S$ is said to be \emph {left abundant} (\emph {right abundant}) if every $\mathcal{L}^{\ast}$-class ($\mathcal{R}^{\ast}$-class) contains an idempotent. Moreover, a semigroup $S$ is called \emph {abundant} if it is both left abundant and right abundant. It is obvious that regular semigroups are abundant, but the converse is not true. For example, Umar \cite{Umar} shown that the semigroup of order-decreasing finite full transformations is abundant but not regular. Many papers have been written describing the abundance of various semigroups.

For a nonempty set $X$, let $\mathcal{T}(X)$ be the full transformation semigroup on $X$ that is, the semigroup under composition of all maps from $X$ into itself. It is well known that $\mathcal{T}(X)$ is a regular semigroup (see \cite{H1}). Transformation semigroups are ubiquitous in semigroup theory because of Cayley's Theorem which states that every semigroup $S$ embeds in some transformation semigroup $\mathcal{T}(X)$ (see \cite[Theorem 1.1.2]{H1}).

Given a nonempty subset $Y$ of $X$, let
$$\mathcal{\overline{T}}(X,Y)=\{\alpha \in \mathcal{T}(X): Y\alpha \subseteq Y\},~~~~\mathcal{T}(X,Y)=\{\alpha \in \mathcal{T}(X): X\alpha \subseteq Y\}.$$
Then $\mathcal{\overline{T}}(X,Y)$ is a subsemigroup of $\mathcal{T}(X)$ and $\mathcal{T}(X,Y)$ is a subsemigroup of $\mathcal{\overline{T}}(X,Y)$. In 1966, Magill \cite{M2} introduced and studied the semigroup $\mathcal{\overline{T}}(X,Y)$. In 1975, Symons \cite{S6} introduced the semigroup $\mathcal{T}(X,Y)$, and also described all automorphisms of $\mathcal{T}(X,Y)$. Recently, $\mathcal{\overline{T}}(X,Y)$ and $\mathcal{T}(X,Y)$ have been studied in a variety of contexts (see \cite{H2,Jin1,Jin2,N,S1,S2,S3,S4,S5,Yan}).

The study of the related combinatorial properties of subsemigroups of finite full transformation semigroup has always been one of the most important topics in the semigroup theory. Many scholars have obtained results (see \cite{A,B1,F,G}). Although they have studied semigroups $\mathcal{\overline{T}}(X,Y)$ and $\mathcal{T}(X,Y)$ from different perspectives, very little research has been found to deal with other literatures have studied other related combinatorial properties of semigroups $\mathcal{\overline{T}}(X,Y)$ and $\mathcal{T}(X,Y)$ except that  Nenthein, Youngkhong and Kemprasit \cite{N} determined the number of all regular elements in $\mathcal{\overline{T}}(X,Y)$ and $\mathcal{T}(X,Y)$.

For $X$, $Y$ and $Z$ are all nonempty sets with $Z\subseteq Y\subseteq X$, the first author \cite{Jin} defined
$$\mathcal{T}(X,Y,Z)=\{\alpha \in \mathcal{T}(X): Y\alpha \subseteq Z \}.$$
Clearly, for each $\alpha ,\beta \in \mathcal{T}(X,Y,Z)$, $Y(\alpha \beta )=(Y\alpha )\beta \subseteq Z\beta \subseteq Y\beta \subseteq Z$ and so $\alpha \beta\in \mathcal{T}(X,Y,Z)$. Therefore, we have $\mathcal{T}(X,Y,Z)$ is a subsemigroup of $\mathcal{T}(X)$, and we call it \emph {the semigroup of transformations with restricted partial range} on $X$. The semigroup $\mathcal{T}(X,Y,Z)$ is a generalization of semigroups $\mathcal{T}(X)$, $\mathcal{\overline{T}}(X,Y)$ and $\mathcal{T}(X,Z)$, that is,

$\bullet $ if $Z=Y$, then $\mathcal{T}(X,Y,Z)=\mathcal{\overline{T}}(X,Y)$;

$\bullet $ if $Y=X$, then $\mathcal{T}(X,Y,Z)=\mathcal{T}(X,Z)$;

$\bullet $ if $Z=Y=X$, then $\mathcal{T}(X,Y,Z)=\mathcal{T}(X)$.

For the case $Z=Y=X$, it is well known that $\mathcal{T}(X,Y,Z)=\mathcal{T}(X)$ is a regular semigroup and so $\mathcal{T}(X,Y,Z)$ is abundant.

For the case $Z=Y\subsetneq X$, Sun \cite{S3} shown the following result.

\begin{lemma}\label{L1.1}{\rm(\cite[Theorem 4.2]{S3})} The semigroup $\mathcal{\overline{T}}(X,Y)$ is abundant.
\end{lemma}

For the case $Z\subsetneq Y=X$. Then $\mathcal{T}(X,Y,Z)=\mathcal{T}(X,Z)$ contains exactly one element if $|Z|=1$. And if $|Z|\geq 2$, Sun \cite{S4} presented the following result.

\begin{lemma}\label{L1.2}{\rm(\cite[Theorem 1]{S4})} The semigroup $\mathcal{T}(X,Z)$ is left abundant but not right abundant.
\end{lemma}

The paper is organized as follows. In Section 2, we characterize the regular elements of the semigroup $\mathcal{T}(X,Y,Z)$, and present a necessary and sufficient condition under which $\mathcal{T}(X,Y,Z)$ is regular. In Section 3, we investigate the abundance of the semigroup $\mathcal{T}(X,Y,Z)$ for the case $Z\subsetneq Y\subsetneq  X$. In Section 4, we compute the cardinalities of $\mathcal{T}(X,Y,Z)$, ${\rm Reg}(\mathcal{T}(X,Y,Z))$ and ${\rm E}(\mathcal{T}(X,Y,Z))$ when $X$ is finite, respectively. All combinatorial formulas in $\mathcal{T}(X,Y,Z)$ also apply to the semigroup $\mathcal{\overline{T}}(X,Y)$ (resp. $\mathcal{T}(X,Z)$ or $\mathcal{T}(X)$). 

Throughout this paper, we always write functions on the right; in particular, this means that for a composition $\alpha \beta $, $\alpha$ is applied first. For any sets $A$ and $B$, we denote by $|A|$ the cardinality of $A$, and write $A\setminus B=\{a\in A: a\notin B\}$. For each $\alpha \in \mathcal{T}(X,Y,Z)$, we denote by $X\alpha $ the range of $\alpha $. And if $A$ is a nonempty subset of $X$ then the restriction of $\alpha$ to the set $A$ is denoted
by $\alpha |_A$. Moreover, for the general background of Semigroup Theory and standard notation, we refer the readers to Howie's book \cite{H1}.

\section{Regularity}
In this section, we characterize the regularity of the semigroup $\mathcal{T}(X,Y,Z)$. First, we describe the regular elements of the semigroup $\mathcal{T}(X,Y,Z)$.

\begin{theorem} \label{T2.1} Let $\alpha \in \mathcal{T}(X,Y,Z)$. Then the following conditions are equivalent:
\\[8pt]
{\rm($i$)} $\alpha \in {\rm Reg}(\mathcal{T}(X,Y,Z))$.
\\[8pt]
{\rm($ii$)} $X\alpha \cap Y\subseteq Z\alpha $.
\\[8pt]
{\rm($iii$)} $X\alpha \cap Y= Z\alpha $.
\end{theorem}

\begin{proof} ($i$) $\Rightarrow $($ii$). Let $\alpha \in {\rm Reg}(\mathcal{T}(X,Y,Z))$. Then exists $\beta \in \mathcal{T}(X,Y,Z)$ such that $\alpha =\alpha \beta \alpha $. For each $x\in X\alpha \cap Y$, we have $x\in Y$ and $x=a\alpha $ for some $a\in X$. Consequently, $x=a\alpha =a\alpha\beta\alpha=x\beta\alpha \in Y\beta\alpha\subseteq Z\alpha $ and so ($ii$) holds.

($ii$)$\Rightarrow $($iii$). It is obvious that $Z\alpha \subseteq Y\alpha \subseteq X\alpha \cap Z\subseteq X\alpha \cap Y$, together with condition ($ii$), we get ($iii$).

($iii$)$\Rightarrow $($i$). Suppose that $X\alpha \cap Y= Z\alpha $, and let
$$X\alpha\cap Y=\{\overline{y_1},\overline{y_2},\cdots ,\overline{y_s}\}.$$
Then exist $z_i\in Z~(i=1,2,\cdots ,s)$ such that $z_i\alpha =\overline{y_i}$. We consider two cases. If $X\alpha \setminus Y=\varnothing$, we define a mapping $\beta :X\rightarrow X$ by
$$x\beta = \left\{ \begin{array}{rl}
z_i         &\mbox{ if $x=\overline{y_i}$ for some $i=1,2,\cdots ,s$  }\\
z_1         &\mbox{ otherwise.} \end{array}  \right. $$
If $X\alpha \setminus Y\neq \varnothing$. Then, for each $x\in X\alpha\setminus Y$, choose and fix $t_x\in \{k\in X: k\alpha =x \}$, and define a mapping $\beta :X\rightarrow X$ by
$$x\beta = \left\{ \begin{array}{rl}
z_i         &\mbox{ if $x=\overline{y_i}$ for some $i=1,2,\cdots ,s$  }\\
t_x           &\mbox{ if $x\in X\alpha\setminus Y$  }\\
z_1         &\mbox{ otherwise.} \end{array}  \right. $$
For both cases, it is easy to verify that $\alpha =\alpha \beta \alpha $ and $\beta \in \mathcal{T}(X,Y,Z)$. Hence ($i$) holds.
\end{proof}

In particular, we take $Z=Y$ (resp., $Y=X$) in Theorem \ref{T2.1}. Then we get the following Corollary \ref{C2.2} (resp., Corollary \ref{C2.3}) which are proved by Nenthein, Youngkhong and Kemprasit \cite[Theorem 2.1]{N} (resp., \cite[Theorem 2.3]{N}).

\begin{corollary}\label{C2.2} Let $\alpha \in \mathcal{\overline{T}}(X,Y)$. Then the following conditions are equivalent:
\\[8pt]
{\rm($i$)} $\alpha \in {\rm Reg}(\mathcal{\overline{T}}(X,Y))$.
\\[8pt]
{\rm($ii$)} $X\alpha \cap Y\subseteq Y\alpha $.
\\[8pt]
{\rm($iii$)} $X\alpha \cap Y= Y\alpha $.
\end{corollary}

\begin{corollary}\label{C2.3} Let $\alpha \in \mathcal{T}(X,Z)$. Then the following conditions are equivalent:
\\[8pt]
{\rm($i$)} $\alpha \in {\rm Reg}(\mathcal{T}(X,Z))$.
\\[8pt]
{\rm($ii$)} $X\alpha \subseteq Z\alpha $.
\\[8pt]
{\rm($iii$)} $X\alpha = Z\alpha $.
\end{corollary}

Nenthein, Youngkhong and Kemprasit presented a necessary and sufficient condition under which $\mathcal{T}(X,Z)$ (resp., $\mathcal{ \overline{T}}(X,Y)$) is regular in \cite{N} that is,

\begin{lemma}\label{L2.4}{\rm(\cite[Corollary 2.2]{N})} $\mathcal{T}(X,Z)$ is a regular semigroup if and only if $|Z|=1$ or $X=Z$.
\end{lemma}

\begin{lemma}\label{L2.5}{\rm(\cite[Corollary 2.4]{N})} $\mathcal{ \overline{T}}(X,Y)$ is a regular semigroup if and only if $|Y|=1$ or $X=Y$.
\end{lemma}

Next, a necessary and sufficient condition for $\mathcal{T}(X,Y,Z)$ to be a regular semigroup can be given as follows:

\begin{theorem}\label{T2.6}$\mathcal{T}(X,Y,Z)$ is a regular semigroup if and only if one of the following statements holds:
\\[8pt]
{\rm($i$)} $|Y|=1.$
\\[8pt]
{\rm($ii$)} $X=Y$ and $|Z|=1.$
\\[8pt]
{\rm($iii$)} $Z=Y=X.$
\end{theorem}

\begin{proof} For $|Y|=1$. It is note that $Z$ be a nonempty subset of $Y$, then $Y=Z$ and so $\mathcal{T}(X,Y,Z)=\mathcal{\overline{T}}(X,Y)$. According to Lemma \ref{L2.5}, we have $\mathcal{T}(X,Y,Z)$ is regular. For $X=Y$ and $|Z|=1$. It is easy to see that $\mathcal{T}(X,Y,Z)=\mathcal{T}(X,Z)$ and so from Lemma \ref{L2.4} it follows that $\mathcal{T}(X,Y,Z)$ is regular. For $Z=Y=X$, we have $\mathcal{T}(X,Y,Z)=\mathcal{T}(X)$ which is regular.

Conversely, suppose that $\mathcal{T}(X,Y,Z)$ is a regular semigroup, and let ($i$), ($ii$) and ($iii$) be not established. Note that $X$, $Y$ and $Z$ are all nonempty sets with $Z\subseteq Y\subseteq X$. To do this, we distinguish three cases:

$\mathbf{Case 1.}$ $Z\subsetneq Y\subsetneq X$. Let $z$ be an element of $Z$, and choose $y\in Y$ such that $y\neq z$. Since $X\setminus Y\neq \varnothing$, we define a mapping $\alpha :X\rightarrow X$ by
$$x\alpha= \left\{ \begin{array}{rl}
z         &\mbox{ if $x\in Y$.  }\\
y         &\mbox{ if $x\in X\setminus Y$.} \end{array}  \right. $$
It is easy to verify that $\alpha \in \mathcal{T}(X,Y,Z)$. However, $X\alpha \cap Y=\{z,y\}\supsetneq \{z\}=Z\alpha $. By Theorem \ref{T2.1}, we immediately deduce that $\alpha $ is not a regular element of $\mathcal{T}(X,Y,Z)$, which contradicts the fact that $\mathcal{T}(X,Y,Z)$ is regular.

$\mathbf{Case 2.}$ $|Z|>1$ and $Z=Y\subsetneq X$. Then $\mathcal{T}(X,Y,Z)=\mathcal{ \overline{T}}(X,Y)$ with $|Y|\neq 1$ and $X\neq Y$. Also, we have $\mathcal{T}(X,Y,Z)$ is not regular by Lemma \ref{L2.5}. This is a contradiction.

$\mathbf{Case 3.}$ $|Z|>1$ and $Z\subsetneq Y= X$. Then $\mathcal{T}(X,Y,Z)=\mathcal{T}(X,Z)$ with $|Z|\neq 1$ and $X\neq Z$. Similar to the above, we have $\mathcal{T}(X,Y,Z)$ is not regular by Lemma \ref{L2.4}. This is a contradiction.
\end{proof}

\section{Abundance}
In this section, we investigate the abundance of the semigroup $\mathcal{T}(X,Y,Z)$ for the case $Z\subsetneq Y\subsetneq X$. The following two lemmas give characterizations of $\mathcal{L}^{\ast}$ and $\mathcal{R}^{\ast}$ that can be found, for instance, in \cite{ Fountain}.

\begin{lemma}{\rm (\cite[Lemma 1.1]{ Fountain})} \label{L3.1}  Let $S$ be a semigroup and $a, b \in S$. Then the following statements are equivalent:
\\[8pt]
{\rm (1)} $(a,b)\in \mathcal{L}^{\ast}$.
\\[8pt]
{\rm (2)} For all $x, y\in S^1$, $ax = ay$ if and only if $bx = by$.
 \end{lemma}

Dually, we have:

\begin{lemma} \label{L3.2}  Let $S$ be a semigroup and $a, b \in S$. Then the following statements are equivalent:
\\[8pt]
{\rm ($i$)} $(a,b)\in \mathcal{R}^{\ast}$.
\\[8pt]
{\rm ($ii$)} For all $x, y\in S^1$, $xa = ya$ if and only if $xb = yb$.
 \end{lemma}

To facilitate the description of the following lemma, we introduce a binary relation $\Lambda$ on $\mathcal{T}(X,Y,Z)$ as follows: For each $\alpha, \beta \in \mathcal{T}(X,Y,Z)$, $(\alpha, \beta)\in \Lambda$ if and only if one of the following statements holds:
\\[8pt]
{\rm ($i$)} $(X\setminus Y)\alpha \cap (Y\setminus Z)= \varnothing$ and $(X\setminus Y)\beta \cap (Y\setminus Z)= \varnothing$.
\\[8pt]
{\rm ($ii$)} $(X\setminus Y)\alpha \cap (Y\setminus Z)\neq \varnothing$ and $(X\setminus Y)\beta \cap (Y\setminus Z)\neq \varnothing$.\\[8pt]
Clearly, $\Lambda$ is an equivalence relation on $\mathcal{T}(X,Y,Z)$.

\begin{lemma}\label{L3.3}  Let $Z\subsetneq Y\subsetneq X$ and $\alpha, \beta \in \mathcal{T}(X,Y,Z)$. Then the following statements hold:
\\[8pt]
{\rm ($i$)} for $|Z|=1$, $(\alpha, \beta)\in \mathcal{L}^{\ast}$ if and only if $(\alpha, \beta)\in \Lambda$ and $X\alpha \cap (X\setminus Y)=X\beta \cap (X\setminus Y)$.
\\[8pt]
{\rm ($ii$)} for $|Z|\geq 2$, $(\alpha, \beta)\in \mathcal{L}^{\ast}$ if and only if $X\alpha =X\beta $.
 \end{lemma}

 \begin{proof} ($i$) Suppose that  $(\alpha, \beta)\in \Lambda$ and $X\alpha \cap (X\setminus Y)=X\beta \cap (X\setminus Y)$. By $|Z|=1$, we say that $Z=\{z_0\}$.  From $(\alpha, \beta)\in \Lambda$, we distinguish two cases:

 {\bf Case 1.} $(X\setminus Y)\alpha \cap (Y\setminus Z)= \varnothing$ and $(X\setminus Y)\beta \cap (Y\setminus Z)= \varnothing$. Clearly, 
\begin{align*}
X\alpha=&Y\alpha\cup (X\setminus Y)\alpha\\
       =&\{z_0\}\cup  \{(X\setminus Y)\alpha \cap [Z\cup (X\setminus Y)]\}    \\
       =&\{z_0\}\cup  \{[(X\setminus Y)\alpha\cap Z]\cup [(X\setminus Y)\alpha \cap (X\setminus Y)]\}\\
                    =&\{z_0\}\cup [(X\setminus Y)\alpha \cap (X\setminus Y)]\\
                    =&\{z_0\}\cup[(X\setminus Y)\alpha \cap (X\setminus Y)]\cup [Y\alpha \cap (X\setminus Y)]~(\mbox{By }Y\alpha \cap (X\setminus Y) \subseteq\\
                     &Z \cap (X\setminus Y)=\varnothing)\\
                    =&\{z_0\}\cup \{[(X\setminus Y)\alpha \cup Y\alpha]  \cap (X\setminus Y)\}\\
                    =&\{z_0\}\cup [X\alpha \cap (X\setminus Y)].
\end{align*}
Similarly, we have $X\beta =\{z_0\}\cup [X\beta \cap (X\setminus Y)]$. Since $X\alpha \cap (X\setminus Y)=X\beta \cap (X\setminus Y)$, we have $X\alpha =X\beta $. This implies that $\alpha$ and $\beta $ are $\mathcal{L}$-related in the full transformation semigroup $\mathcal{T}(X)$ (see \cite[page 63]{H1}). Hence $(\alpha, \beta)\in \mathcal{L}^{\ast}$.
 
 {\bf Case 2.} $(X\setminus Y)\alpha \cap (Y\setminus Z)\neq \varnothing$ and $(X\setminus Y)\beta \cap (Y\setminus Z)\neq \varnothing$. For each $\eta ,\theta \in \mathcal{T}^1(X,Y,Z)$, we consider the following three subcases:

 Case 2.1. $\eta=1$ and $\theta=1$. Clearly, $(\alpha, \beta)\in \mathcal{L}^{\ast}$.

 Case 2.2. $\eta=1$ and $\theta \neq 1$. Then $\theta \in \mathcal{T}(X,Y,Z)$ and so $Y\theta \subseteq Z=\{z_0\}$. Let $\gamma \eta=\gamma \theta$ ($\gamma \in \{\alpha,\beta\}$). Then $\gamma =\gamma \theta$ and so $x\theta =x$ for all $x\in X\gamma$. This means that $(X\setminus Y)\gamma \cap (Y\setminus Z)=\varnothing$ (If not, there exist $b_{\gamma}\in X\setminus Y$ and $y_{\gamma}\in Y\setminus Z$ such that $y_{\gamma}=b_{\gamma}\gamma \in X\gamma$. Then $y_{\gamma}=y_{\gamma}\theta\in Z$, this contradicts the condition that $y_{\gamma}\in Y\setminus Z$). This is a contradiction.

 Case 2.3. $\eta\neq 1$ and $\theta \neq 1$. That is, $\eta, \theta \in \mathcal{T}(X,Y,Z)$. Then $Y\eta =\{z_0\}=Y\theta$ and so $\eta |_Y=\theta |_Y$. Therefore,
\begin{align*}
\alpha \eta =\alpha\theta &\Leftrightarrow \eta |_{X\alpha }=\theta |_{X\alpha }   \nonumber\\
&\Leftrightarrow \eta |_{X\alpha \cap Y}=\theta |_{X\alpha \cap Y}\mbox{ and }\eta |_{X\alpha \cap (X\setminus Y)}=\theta |_{X\alpha \cap (X\setminus Y)} \nonumber\\
&\Leftrightarrow \eta |_{X\beta \cap Y}=\theta |_{X\beta \cap Y}\mbox{ and }\eta |_{X\beta \cap (X\setminus Y)}=\theta |_{X\beta \cap (X\setminus Y)}\\
&\Leftrightarrow \eta |_{X\beta }=\theta |_{X\beta } \nonumber\\
&\Leftrightarrow \beta\eta =\beta\theta. \label{eq:myalign}
\end{align*}
By Lemma \ref{L3.1} we conclude that $(\alpha, \beta)\in \mathcal{L}^{\ast}$.

Conversely, suppose that $(\alpha, \beta)\in \mathcal{L}^{\ast}$  such that $(\alpha,\beta)\notin \Lambda$ or $X\alpha \cap (X\setminus Y)\neq X\beta \cap (X\setminus Y)$. We distinguish two cases: 

 {\bf Case 1.} $(\alpha,\beta)\notin \Lambda$. Then we have ($(X\setminus Y)\alpha \cap (Y\setminus Z)=\varnothing$ and $(X\setminus Y)\beta \cap (Y\setminus Z)\neq \varnothing$) or ($(X\setminus Y)\alpha \cap (Y\setminus Z)\neq \varnothing$ and $(X\setminus Y)\beta \cap (Y\setminus Z)=\varnothing$). By symmetry, let $(X\setminus Y)\alpha \cap (Y\setminus Z)=\varnothing$ and $(X\setminus Y)\beta \cap (Y\setminus Z)\neq \varnothing$. Define two mappings $\eta :X\rightarrow X$ and $\theta :X\rightarrow X$ by $\eta =1$ and
$$x\theta = \left\{ \begin{array}{rl}
x         &\mbox{ if $x\in X\alpha$ }\\
z_0         &\mbox{ if $x\notin X\alpha$.} \end{array}  \right.$$
Clearly, $\theta\in \mathcal{T}(X,Y,Z)$ and $\alpha\eta =\alpha\theta$. Howerver, $\beta\eta \neq\beta\theta$. This contradicts the fact that $(\alpha, \beta)\in \mathcal{L}^{\ast}$.

{\bf Case 2.} $X\alpha \cap (X\setminus Y)\neq X\beta \cap (X\setminus Y)$. Then exists $a\in X\beta \cap (X\setminus Y)$ such that $a\notin X\alpha \cap (X\setminus Y)$ and so $a_0\beta =a$ for some $a_0\in X$ and $x\alpha \neq a$ for all $x\in X$. In fact, $a_0\in X\setminus Y$ (If not, $a=a_0\beta \in Z$, this contradicts the fact that $a\in X\setminus Y$). We consider two cases. If $|X\setminus Y|=1$. It is clear that $X\setminus Y=\{a\}$. Define two mappings $\eta :X\rightarrow X$ and $\theta :X\rightarrow X$ by $X\eta =z_0$ and
$$x\theta = \left\{ \begin{array}{rl}
z_0         &\mbox{ if $x\in Y$ }\\
a_0         &\mbox{ if $x\in X\setminus Y$.} \end{array}  \right.$$
If $|X\setminus Y|\geq 2$. Define two mappings $\eta :X\rightarrow X$ and $\theta :X\rightarrow X$ by
$$x\eta = \left\{ \begin{array}{rl}
z_0         &\mbox{ if $x\in Y\cup \{a\}$ }\\
a_0         &\mbox{ if $x\in X\setminus (Y\cup \{a\})$} \end{array}  \right. \mbox { and~~~~~~}x\theta = \left\{ \begin{array}{rl}
z_0         &\mbox{ if $x\in Y$ }\\
a_0         &\mbox{ if $x\in X\setminus Y$.} \end{array}  \right. $$
For both cases, we have $\eta ,\theta \in \mathcal{T}(X,Y,Z)$ and $\alpha \eta =\alpha \theta$. However, $$a_0\beta \eta =a\eta =z_0\neq a_0=a\theta =a_0\beta \theta$$
 and so $\beta \eta\neq \beta\theta$. This contradicts the fact that $(\alpha, \beta)\in \mathcal{L}^{\ast}$. 
 
Hence, $(\alpha,\beta)\in \Lambda$ and $X\alpha \cap (X\setminus Y)= X\beta \cap (X\setminus Y)$.

($ii$) Let $X\alpha =X\beta$. This implies that $\alpha ,\beta $ are $\mathcal{L}$-related in the full transformation semigroup $\mathcal{T}(X)$. Hence $(\alpha, \beta)\in \mathcal{L}^{\ast}$.

Conversely, suppose that $(\alpha, \beta)\in \mathcal{L}^{\ast}$ and $X\alpha \neq X\beta$. Then exists $a\in X\beta $ such that $a \notin X\alpha $ and so $a_0\beta =a$ for some $a_0\in X$ and $x\alpha \neq a$ for all $x\in X$.  Note that $|Z|\geq 2 $ and $|X|\geq 4$. Then we can take distinct $z_1, z_2\in Z$, and choose nonempty subsets $X_1$, $X_2$ of $X$ with $|X_i|\geq 2$ ($i=1,2$) such that $X$  is a disjoint union of $X_1$ and $X_2$. Define two mappings $\eta :X\rightarrow X$ and $\theta :X\rightarrow X$ by
$$x\eta = \left\{ \begin{array}{rl}
z_1         &\mbox{ if $x\in X_1\cup \{a\}$ }\\
z_2         &\mbox{ if $x\in X_2\setminus \{a\}$} \end{array}  \right. \mbox { and~~~~~~}x\theta = \left\{ \begin{array}{rl}
z_1         &\mbox{ if $x\in X_1\setminus \{a\}$ }\\
z_2         &\mbox{ if $x\in X_2\cup \{a\}$.} \end{array}  \right. $$
Clearly, $\eta ,\theta \in \mathcal{T}(X,Y,Z)$ and $\alpha \eta =\alpha \theta$. However, $$a_0\beta \eta =a\eta =z_1\neq z_2=a\theta =a_0\beta \theta$$ and so $\beta \eta\neq \beta\theta$. This contradicts the fact that $(\alpha, \beta)\in \mathcal{L}^{\ast}$. Hence $X\alpha =X\beta$.
\end{proof}

A necessary and sufficient condition for $\alpha\in \mathcal{T}(X,Y,Z)$ to be an idempotent can be given as follows:

\begin{lemma} \label{L3.4}  Let $\alpha \in \mathcal{T}(X,Y,Z)$. Then $\alpha$ is an idempotent if and only if the following statements hold:
\\[8pt]
{\rm($i$)} $X\alpha \subseteq Z\cup(X\setminus Y)$.
\\[8pt]
{\rm ($ii$)} $t\alpha =t$ for all $t\in X\alpha$.
 \end{lemma}

 \begin{proof} Suppose that $X\alpha \subseteq Z\cup (X\setminus Y)$ and $t\alpha =t$ for all $t\in X\alpha$. For each $x\in X$, there exists $t\in X\alpha $ such that $x\alpha =t$. Then $x\alpha ^2=(x\alpha)\alpha =t\alpha =t=x\alpha$. Hence $\alpha$ is an idempotent.

Conversely, suppose that $\alpha$ is an idempotent, and let ($i$) or ($ii$)  do not hold. To do this, we distinguish two cases:

$\mathbf{Case 1}$. ($i$) not holds and ($ii$) holds. Then  there exists $y\in X\alpha $ such that $y\in Y\setminus Z$ and so $y\alpha =y\notin Z$. This is a contradiction.

$\mathbf{Case 2}$. ($ii$) not holds. There exists $t_0\in X\alpha $ such that $t_0\alpha \neq t_0$.  Note that $x_0\alpha =t_0$ for some $x_0\in X$. Then $x_0\alpha ^2=(x_0\alpha)\alpha =t_0\alpha \neq t_0=x_0\alpha$, which contradicts the fact that $\alpha$ is an idempotent.
\end{proof}

\begin{lemma} \label{L3.5} Let $Z\subsetneq Y\subsetneq X$.  Then not each $\mathcal{L}^{\ast}$-class of $\mathcal{T}(X,Y,Z)$ contains an idempotent.
 \end{lemma}

 \begin{proof}  Let $f\in \mathcal{T}(X,Y,Z)$ such that $(X\setminus Y)f\cap (Y\setminus Z)\neq \varnothing$. Next, we prove that the $\mathcal{L}^{\ast}$-class $\mathcal{L}^{\ast}_{f}$ containing $f$ has no idempotents. Assume that $(f,e)\in \mathcal{L}^{\ast}$ for some  idempotent $e\in \mathcal{T}(X,Y,Z)$, then two cases are considered as follows:

  {\bf Case 1.} $|Z|=1$. Since Lemma \ref{L3.3} (1) it follows that $(X\setminus Y)e\cap (Y\setminus Z)\neq \varnothing$ and so $Xe\cap (Y\setminus Z)\neq \varnothing$. 
 
 {\bf Case 2.} $|Z|\geq 2$. From Lemma \ref{L3.3} (2) it follows that $Xe=Xf$ and so $Xe\cap (Y\setminus Z)\neq \varnothing$.
 
 However, we have $Xe\subseteq Z\cup(X\setminus Y)$ since Lemma \ref{L3.4} (1). Note that $Z\subsetneq Y\subsetneq X$, then $Xe\cap (Y\setminus Z)= \varnothing$. This is a contradiction.
\end{proof}

After that, we consider the $\mathcal{R}^{\ast}$-relation. Let $\pi _{\alpha}$ be the partition of $X$
induced by $\alpha \in \mathcal{T}(X,Y,Z)$, namely,
$$\pi _{\alpha}=\{x\alpha ^{-1}: x\in X\alpha\}.$$

\begin{lemma}\label{L3.6} Let $Z\subsetneq Y\subsetneq X$ and $\alpha, \beta \in \mathcal{T}(X,Y,Z)$. Then $(\alpha, \beta)\in \mathcal{R}^{\ast}$ if and only if $\pi _{\alpha}=\pi _{\beta}$.
 \end{lemma}

 \begin{proof} Let $\pi _{\alpha}=\pi _{\beta}$. This implies that $\alpha ,\beta $ are $\mathcal{R}$-related in the full transformation semigroup $\mathcal{T}(X)$ (see \cite[page 63]{H1}). Hence $(\alpha, \beta)\in \mathcal{R}^{\ast}$.

Conversely, suppose that $(\alpha, \beta)\in \mathcal{R}^{\ast}$ and $x_1\alpha  = x_2\alpha$ for some distinct $x_1, x_2\in X$. We show that $x_1\beta  = x_2\beta $. There are three cases to be considered.

$\mathbf{Case~1}$. $x_1, x_2\in Z$. Define two mappings $\eta :X\rightarrow X$ and $\theta :X\rightarrow X$ by
$$x\eta = \left\{ \begin{array}{rl}
x_1         &\mbox{ if $x\in Y$ }\\
x         &\mbox{ if $x\in X\setminus Y$} \end{array}  \right. \mbox { and~~~~~~}x\theta = \left\{ \begin{array}{rl}
x_2         &\mbox{ if $x\in Y$ }\\
x         &\mbox{ if $x\in X\setminus Y$.} \end{array}  \right. $$
Clearly, $\eta ,\theta \in \mathcal{T}(X,Y,Z)$ and $\eta\alpha = \theta \alpha$. Then $\eta\beta =\theta\beta $ and so $x_1\beta =Y\eta\beta =Y\theta \beta =x_2\beta $.

$\mathbf{Case~2}$. $x_1, x_2\in X\backslash Z$. Choose and fix $z_0\in Z$. Define two mappings $\eta :X\rightarrow X$ and $\theta :X\rightarrow X$ by
$$x\eta = \left\{ \begin{array}{rl}
z_0         &\mbox{ if $x\in Y$ }\\
x_1         &\mbox{ if $x\in X\setminus Y$} \end{array}  \right. \mbox { and~~~~~~}x\theta = \left\{ \begin{array}{rl}
z_0         &\mbox{ if $x\in Y$ }\\
x_2       &\mbox{ if $x\in X\setminus Y$.} \end{array}  \right. $$
Clearly, $\eta ,\theta \in \mathcal{T}(X,Y,Z)$ and $\eta\alpha = \theta \alpha$. Then $\eta\beta =\theta\beta $ and so $x_1\beta =(X\setminus Y)\eta\beta =(X\setminus Y)\theta \beta =x_2\beta $.

$\mathbf{Case~3}$. $x_1\in Z$ and $x_2\in X\backslash Z$. Define two mappings $\eta :X\rightarrow X$ and $\theta :X\rightarrow X$ by $X\eta =x_1$ and
$$x\theta = \left\{ \begin{array}{rl}
x_1         &\mbox{ if $x\in Y$ }\\
x_2         &\mbox{ if $x\in X\setminus Y$.} \end{array}  \right. $$
Clearly, $\eta ,\theta \in \mathcal{T}(X,Y,Z)$ and $\eta\alpha = \theta \alpha$. Then $\eta\beta =\theta\beta $ and so $x_1\beta =(X\setminus Y)\eta\beta =(X\setminus Y)\theta \beta =x_2\beta $.

For both cases, we have $\pi _{\alpha }\subseteq \pi _{\beta }$. Dually, we may show that $\pi _{\beta }\subseteq \pi _{\alpha }$. Consequently, $\pi _{\alpha }=\pi _{\beta }$.
\end{proof}

\begin{lemma} \label{L3.7} Let $Z\subsetneq Y\subsetneq X$. Then the following statements hold:
\\[8pt]
{\rm($i$)} for $|Z|=1$, each $\mathcal{R}^{\ast}$-class of $\mathcal{T}(X,Y,Z)$ contains an idempotent.
\\[8pt]
{\rm($ii$)} for $|Z|\geq 2$, not each $\mathcal{R}^{\ast}$-class of $\mathcal{T}(X,Y,Z)$ contains an idempotent.
 \end{lemma}

 \begin{proof} ($i$) Let $\alpha \in \mathcal{T}(X,Y,Z)$. Then exists an index set $I$ such that $\pi _{\alpha}=\{A_i: i\in I\}$.  Note that $Y\alpha \subseteq Z$ and $|Z|=1$, there exists $i\in I$ such that $Y\subseteq A_i$. Take $z_0\in Z$ and $a_j\in A_j$ for all $j\in I\setminus \{i\}$. Define a mapping $e:X\rightarrow X$ by
$$xe= \left\{ \begin{array}{rl}
z_0         &\mbox{ if $x\in A_i$ }\\
a_j         &\mbox{ if $x\in A_j$ for all $j\in I\setminus \{i\}$. } \end{array}  \right. $$
Clearly, $e\in \mathcal{T}(X,Y,Z)$ is an idempotent and $\pi _{\alpha }=\pi _{e}$. By Lemma \ref{L3.6}, we have  $(\alpha, e)\in \mathcal{R}^{\ast}$. Hence each $\mathcal{R}^{\ast}$-class of $\mathcal{T}(X,Y,Z)$ contains an idempotent.

($ii$) By $|Z|\geq 2$, we can take distinct $z_1 ,z_2\in Z$. Define $f\in \mathcal{T}(X,Y,Z)$ such that $Zf =z_1$ and $(Y\setminus Z)f =z_2$. Then $Z\subseteq A_i $ and $(Y\setminus Z)\subseteq A_j$ for some distinct $A_i, A_j\in \pi _{f}=\{A_j: j\in J\}$ where $J$ be some index set. We assert that the $\mathcal{R}^{\ast}$-class $\mathcal{R}^{\ast}_{f}$ containing $f$ has no idempotents. Indeed, if $(f,e)\in \mathcal{R}^{\ast}$ for some idempotent $e\in \mathcal{T}(X,Y,Z)$. Then, by Lemma \ref{L3.6} it follows that $\pi _{e}=\pi _{f}$. According to Lemma \ref{L3.4}, $|A_je\cap A_j|=1$ and so $(Y\setminus Z)e=A_je\in A_j$.  Note that $Z\subseteq A_i$ and $A_i\cap A_j=\varnothing$. Then $(Y\setminus Z)e\cap Z=\varnothing$. This contradicts the fact that $(Y\setminus Z)e\subsetneq Ye\subseteq Z$.
\end{proof}

By Lemmas \ref{L3.5} and \ref{L3.7}, we obtain the main result in this section.

\begin{theorem} \label{T3.8}   Let $Z\subsetneq Y\subsetneq X$. Then the following statements hold:
\\[8pt]
{\rm ($i$)} for $|Z|=1$, the semigroup $\mathcal{T}(X,Y,Z)$ is  right abundant.
\\[8pt]
{\rm ($ii$)} for $|Z|\geq 2$, the semigroup $\mathcal{T}(X,Y,Z)$ is neither left
abundant nor right abundant.
\end{theorem}

As a consequence of Lemma \ref{L1.1}, Lemma \ref{L1.2} and Theorem \ref{T3.8}, we have the following conclusion.

\begin{corollary}\label{C3.9} {\rm (I)} for $Z=Y=X$, the semigroup $\mathcal{T}(X,Y,Z)=\mathcal{T}(X)$ is abundant.
\\[8pt]
{\rm (II)} for $Z\subsetneq Y=X$,

{\rm ($i$)} $|Z|=1$, the semigroup $\mathcal{T}(X,Y,Z)=\mathcal{T}(X,Z)$ is abundant.

{\rm ($ii$)} $|Z|\geq 2$, the semigroup $\mathcal{T}(X,Y,Z)=\mathcal{T}(X,Z)$ is left abundant but not right abundant.
\\[8pt]
{\rm (III)} for $Z=Y\subsetneq X$, the semigroup $\mathcal{T}(X,Y,Z)=\mathcal{\overline{T}}(X,Y)$ is abundant.
\\[8pt]
{\rm (IV)} for $Z\subsetneq Y\subsetneq X$,

{\rm ($i$)} $|Z|=1$, the semigroup $\mathcal{T}(X,Y,Z)$ is  right abundant.

{\rm ($ii$)} $|Z|\geq 2$, the semigroup $\mathcal{T}(X,Y,Z)$ is neither left
abundant nor right abundant.
\end{corollary}

\section{Some combinatorial results}

The Stirling number of the second kind $S(n,r)$ counts the number of partitions of a set of $n$ elements into $r$ indistinguishable boxes in which no box is empty. Recall that the number of ways that $r$ objects can be chosen from $n$ distinct
objects written $\binom {n}{r}$ is given by
$$\binom {n}{r}=\frac{n!}{(n-r)!r!}.$$
It is shown in \cite[Theorem 8.26]{B2} that
 $$S(n,r)=\frac{1}{r!}\sum _{i=0}^r(-1)^i\binom {r}{i}(r-i)^n.$$
for integers $n$ and $r$ with $0\leq r\leq n$. In particular, $S(p,0)=0~(p\geq 1)$ and $S(0,0)=1$. B\'{o}na \cite{B3} also presented a formula related Stirling number, that is,

\begin{lemma}{\rm (\cite[page 32]{B3})} \label{L4.1}  Let $m,k\in \mathbb{N}$ such that $1\leq k\leq m$. Then
$$\sum _{r=1}^k\binom {k}{r}r!S(m,r)=k^m.$$
 \end{lemma}

\begin{lemma} \label{L4.2}  Let $|X|=n$, $|Y|=m$ and $|Z|=k$. Then, for each $r\in \mathbb{N}$ with $1\leq r\leq k$,
\begin{equation} \label{E4.1}
  |\{\alpha \in \mathcal{T}(X,Y,Z): |Y\alpha |=r\}|=\binom {k}{r}r!S(m,r)n^{n-m}.
\end{equation}
\end{lemma}

\begin{proof} Let $Z ^{\prime}$ be a nonempty subset of $Z$ with $|Z ^{\prime}|=r$, we have $1\leq r\leq k$ since $|Z|=k$. It is easy to see that the number of mappings $\alpha :X\rightarrow X$ such that $Y\alpha =Z ^{\prime}$ and $(X\setminus Y)\alpha \subseteq X$ is $r!S(m,r)n^{n-m}$, that is, $$|\{\alpha \in \mathcal{T}(X,Y,Z): Y\alpha =Z ^{\prime}\}|=r!S(m,r)n^{n-m}.$$
Consequently,  Equation (\ref{E4.1}) holds for each $r\in \mathbb{N}$ with $1\leq r\leq k$.
\end{proof}

\begin{theorem} \label{T4.3}  Let $|X|=n$, $|Y|=m$ and $|Z|=k$. Then
\begin{equation} \label{E4.2}
|\mathcal{T}(X,Y,Z)|=\sum _{r=1}^k\binom {k}{r}r!S(m,r)n^{n-m}=k^mn^{n-m}.
\end{equation}
\end{theorem}

\begin{proof} According to Lemma \ref{L4.2}, we have
$$|\{\alpha \in \mathcal{T}(X,Y,Z): |Y\alpha |=r\}|=\binom {k}{r}r!S(m,r)n^{n-m}$$
for each $r\in \mathbb{N}$ with $1\leq r\leq k$. Then $|\mathcal{T}(X,Y,Z)|=\sum _{r=1}^k\binom {k}{r}r!S(m,r)n^{n-m}$ by the summing up over all $r$. Moreover, from Lemma \ref{L4.1} it follows that $\sum _{r=1}^k\binom {k}{r}r!S(m,r)n^{n-m}=k^mn^{n-m}$. Hence,  Equation (\ref{E4.2}) as required.
\end{proof}

Since Theorem \ref{T4.3}, we obtain the following corollary which appears in \cite[page 311]{N}.

\begin{corollary} \label{C4.4}  Let $|X|=n$, $|Y|=m$ and $|Z|=k$. Then
\\[8pt]
{\rm ($i$)} $|\mathcal{\overline{T}}(X,Y)|=\sum _{r=1}^m\binom {m}{r}r!S(m,r)n^{n-m}=m^mn^{n-m}.$
\\[8pt]
{\rm ($ii$)} $|\mathcal{T}(X,Z)|=\sum _{r=1}^k \binom {k}{r}r!S(n,r)=k^n.$
\\[8pt]
{\rm ($iii$)} $|\mathcal{T}(X)|=\sum _{r=1}^n\binom {n}{r}r!S(n,r)=n^n.$
\end{corollary}

Next, we determine the  number of all regular elements in the semigroup $\mathcal{T}(X,Y,Z)$ when $X$ is finite.

\begin{theorem} \label{T4.5}  Let $|X|=n$, $|Y|=m$ and $|Z|=k$. Then
\begin{equation} \label{E4.3}
|{\rm Reg}(\mathcal{T}(X,Y,Z))|=\sum _{r=1}^k\binom {k}{r}r!S(k,r)r^{m-k}(n-m+r)^{n-m}.
\end{equation}
\end{theorem}

\begin{proof} For each $\alpha \in {\rm Reg}(\mathcal{T}(X,Y,Z))$, we have $X\alpha \cap Y=Z\alpha \subseteq Y\alpha \subseteq Z$
 by Theorem \ref{T2.1}. Then exists a nonempty subset $Z^{\prime }$ of $Z$ with $|Z^{\prime }|=r$ such that $Z\alpha=X\alpha \cap Y=Z^{\prime }$.
 Clearly, $(Y\setminus Z)\alpha \subseteq Y\alpha \subseteq X\alpha \cap Z\subseteq X\alpha \cap Y=Z^{\prime }$ and so
\begin{equation} \label{E4.4}
(Y\setminus Z)\alpha \subseteq Z^{\prime }.
\end{equation}
We can also assert
\begin{equation} \label{E4.5}
(X\setminus Y)\alpha \subseteq Z^{\prime }\cup (X\setminus Y)
\end{equation}
(If not, there  exists some $y\in X\setminus Y$ such that $y\alpha \in Y\setminus Z^{\prime }$, then $y\alpha \in X\alpha \cap Y=Z^{\prime }$.  This is a contradiction).  Conversely, if a mapping $\alpha \in \mathcal{T}(X,Y,Z)$  satisfies $Z\alpha =Z^{\prime}$, formulas (\ref{E4.4}) and (\ref{E4.5}), it is easy to see that
$$X\alpha \cap Y=[Z\cup (Y\setminus Z)\cup (X\setminus Y)]\alpha \cap Y\subseteq [Z^{\prime }\cup (X\setminus Y)]\cap Y= Z^{\prime }= Z\alpha$$
since $Z^{\prime }\subseteq Z\subseteq Y\subseteq X$. Then, by Theorem \ref{T2.1}, we have $ \alpha \in {\rm Reg}(\mathcal{T}(X,Y,Z))$ and $Z\alpha =Z^{\prime }$. Hence, for each nonempty set $Z^{\prime }\subseteq Z$, we have
\begin{align*}
&\{\alpha \in {\rm Reg}(\mathcal{T}(X,Y,Z)): Z\alpha =Z^{\prime }\}  \nonumber\\
=&\{\alpha \in \mathcal{T}(X,Y,Z): \alpha \mbox {  satisfies }Z\alpha =Z^{\prime }, \mbox { formulas } (\ref{E4.4}) \mbox { and } (\ref{E4.5})\}. \label{eq:myalign}
\end{align*}
It follows that the number of maps $\alpha \in \mathcal{T}(X,Y,Z)$  satisfying $Z\alpha =Z^{\prime}$, formulas (\ref{E4.4}) and (\ref{E4.5}) is $r!S(k,r)r^{m-k}(n-m+r)^{n-m}$ since $|Z^{\prime }\cup (X\setminus Y)|=|X\setminus Y|+|Z^{\prime }|=n-m+r$, that is,
$$|\{\alpha \in {\rm Reg}(\mathcal{T}(X,Y,Z)): Z\alpha =Z^{\prime }\}|=r!S(k,r)r^{m-k}(n-m+r)^{n-m}.$$
Consequently, for each $r\in \mathbb{N}$ with $1\leq r\leq k$,
$$|\{\alpha \in {\rm Reg}(\mathcal{T}(X,Y,Z)): |Z\alpha |=r\}|=\binom {k}{r}r!S(k,r)r^{m-k}(n-m+r)^{n-m}$$
and so  Equation (\ref{E4.3}) holds by the summing up over all $r$.
\end{proof}

Since Theorem \ref{T4.5}, we obtain the following corollary which appears in \cite[Theorem 2.6 and Theorem 2.7]{N}.

\begin{corollary} \label{C4.6}  Let $|X|=n$, $|Y|=m$ and $|Z|=k$. Then
\\[8pt]
{\rm ($i$)} $|{\rm Reg}(\mathcal{\overline{T}}(X,Y))|=\sum _{r=1}^{m}\binom {m}{r}r!S(m,r)(n-m+r)^{n-m}.$
\\[8pt]
{\rm ($ii$)} $|{\rm Reg}(\mathcal{T}(X,Z))|=\sum _{r=1}^{k}\binom {k}{r}r!S(k,r)r^{n-k}.$
\end{corollary}

Moreover, we compute the cardinality of ${\rm E}(\mathcal{T}(X,Y,Z))$.

\begin{theorem} \label{T4.7}  Let $|X|=n$, $|Y|=m$ and $|Z|=k$. Then
\begin{equation} \label{E4.6}
|{\rm E}(\mathcal{T}(X,Y,Z))|=\sum _{r=1}^{n-m+k}\sum _{i=\max\{1,m-n+r\}}^{\min\{k,r\}}\binom {k}{i}\binom {n-m}{r-i}i^{m-i}r^{n-m-r+i}.
\end{equation}
\end{theorem}

\begin{proof} Define an idempotent $\alpha$ with $|X\alpha |=r$, we have to choose a $r$-element set $X\alpha $, then exists $i\in \mathbb{N}$ such that $|X\alpha \cap Z|=i$ and $|X\alpha \cap (X\setminus Y)|=r-i$ by Lemma \ref{L3.4} (There are $\binom {k}{i}\binom {n-m}{r-i}$ different ways). Also, we have to define a mapping  $\varphi: X\setminus X\alpha \rightarrow X\alpha $ such that $\varphi(Y\setminus X\alpha)\subseteq Z$ and $\varphi((X\setminus Y)\setminus X\alpha) \subseteq X\alpha $ in an arbitrary way (This can be done in $i^{m-i}r^{n-m-r+i}$ different ways).  Note that $i$ meets $1\leq i\leq k$ and $0\leq r-i\leq n-m$. Then $\max\{1,m-n+r\}\leq i\leq \min\{k,r\}$. Hence
$$|\{\alpha \in {\rm E}(\mathcal{T}(X,Y,Z)):|X\alpha |=r\}|=\sum _{i=\max\{1,m-n+r\}}^{\min\{k,r\}}\binom {k}{i}\binom {n-m}{r-i}i^{m-i}r^{n-m-r+i}$$
by summing up over all $i$.  Note that
$$1\leq r=|X\alpha|\leq |Y\alpha |+|(X\setminus Y)\alpha|\leq |Z|+|X\setminus Y|=n-m+k.$$
Therefore  Equation (\ref{E4.6}) is now obtained by summing up over all $r$.
\end{proof}

Since Theorem \ref{T4.7}, we obtain the following corollary.
\begin{corollary} \label{C4.8}  Let $|X|=n$, $|Y|=m$ and $|Z|=k$. Then
\\[8pt]
{\rm($i$)} $|{\rm E}(\mathcal{\overline{T}}(X,Y))|=\sum _{r=1}^{n}\sum _{i=\max\{1,m-n+r\}}^{\min\{m,r\}}\binom {m}{i}\binom {n-m}{r-i}i^{m-i}r^{n-m-r+i}.$
\\[8pt]
{\rm($ii$)} $|{\rm E}(\mathcal{T}(X,Z))|=\sum _{r=1}^{k}\binom {k}{r}r^{n-r}.$
\\[8pt]
{\rm($iii$)} $|{\rm E}(\mathcal{T}(X))|=\sum _{r=1}^{n}\binom {n}{r}r^{n-r}.$
\end{corollary}

\section*{Acknowledgements}
The authors would like to thank the anonymous reviewers for their comments that helped improve this paper. This work was supported by the doctoral research start-up fund of Guiyang University (GYU-KY-2023).

\noindent Accepted: 7.2.2023

\end{document}